\documentclass{amsart}%
\usepackage{enumerate}
\usepackage{amsmath}
\usepackage{amsfonts}
\usepackage{amssymb}
\usepackage{graphicx}%
\providecommand{\U}[1]{\protect\rule{.1in}{.1in}}
\newtheorem{theorem}{Theorem}

\newtheorem{corollary}[theorem]{Corollary}

\theoremstyle{example}
\newtheorem{example}[theorem]{Example}

\newtheorem{lemma}[theorem]{Lemma}

\theoremstyle{remark}
\newtheorem{remark}[theorem]{Remark}

\begin{document}
\title[Multiplication operators]{Multiplication operators on vector-valued function spaces }
\author{H\"{u}lya Duru}
\address{Istanbul University, Faculty of Science, Mathematics Department, Vezneciler-Istanbul, 34134, Turkey}
\email{hduru@istanbul.edu.tr}
\author{Arkady Kitover}
\address{Mathematics Department,  Philadelphia Community College, Philadelphia, PA 19130}
\email{akitover@ccp.edu}
\author{Mehmet Orhon}
\address{Mathematics and Statistics Department, University of New Hampshire, Durham, NH 03824}
\email{mo@unh.edu}
\thanks{The first author was supported by the Scientific Projects Coordination Unit of Istanbul University, Project number 3952}
\date{April 5, 2011}

\begin{abstract}
Let $E$ be a Banach function space on a probability measure space $(\Omega
,\Sigma,\mu).$ Let $X$ be a Banach space and $E(X)$ be the associated
K\"{o}the-Bochner space. An operator on $E(X)$ is called a multiplication
operator if it is given by multiplication by a function in $L^{\infty}(\mu).$
In the main result of this paper, we show that an operator $T$ on $E(X)$ is a
multiplication operator if and only if $T$ commutes with $L^{\infty}(\mu)$ and
leaves invariant the cyclic subspaces generated by the constant vector-valued
functions in $E(X).$ As a corollary we show that this is equivalent to $T$
satisfying a functional equation considered by Calabuig, Rodr\'{i}guez,
S\'{a}nchez-P\'{e}rez in [3].

\end{abstract}
\maketitle

\section{Introduction}

Our paper is motivated by the following question which continues to attract
considerable attention. Let $M$ be a linear space of (scalar or vector-valued)
functions on some set, let an algebra of scalar functions on the same set be
given, and let $T:M\rightarrow M$ be a linear operator. What properties of $T$
guarantee that it is an operator of multiplication by some function in the
given algebra of scalar functions? The answer often can be expressed in terms
of $T$-invariant linear subspaces of $M$. On the other hand, it is well known
(see \cite{AVK}) that a linear operator on a vector lattice can be
represented as a multiplication operator if and only if it is regular and band
preserving. If the algebra of such operators is rich enough (e.g., if $M$ is a
Banach lattice with a quasi-interior point) then we have another criterion: a
regular operator is a multiplication operator if and only if it commutes with
all regular band preserving operators. The first approach works for Banach
$C(K)$-modules as well. If $M$ is a Banach $C(K)$-module then a (bounded linear) 
operator $T$ on $M$ is induced by multiplication by a function from the
closure of $C(K)$ in the weak-operator topology on $\mathcal{L}(M)$ if and
only if $T$ leaves invariant every cyclic subspace $M(x)$ where $x\in M$ and
$M(x)=cl\{ax:\;a\in C(K)\}$ (\cite[Theorem 6.2]{AAK}, \cite[Theorem 7]{HO}). But
this condition is quite strong and not so easy to verify. Inspired by the
paper of Calabuig, Rodr\'{i}guez, and S\'{a}nchez-P\'{e}rez \cite{CRS} we looked at the
case when the operator leaves invariant only some special cyclic subspaces.
More precisely, we look at the case when the Banach $C(K)$-module (with
$C(K)=L^{\infty}(\mu)$) can be represented as a Banach space of vector-valued
measurable functions and the special cyclic subspaces are generated by the
constant vector-valued functions. The condition that an operator leaves such
subspaces invariant is in general not sufficient for being a multiplication
operator and we have to require additionally that the restriction of the
operator on such a subspace commutes with the restrictions of multiplication
operators to the same subspace (or more strongly, that the operator commutes
with the multiplication operators everywhere) . The techniques developed in
the paper allowed us to show that these two conditions together already become
sufficient to identify multiplication operators on a very broad class of
spaces of vector-valued measurable functions. Thus we considerably extend the
results in \cite{CRS}.

\section{Spaces of vector-valued continuous functions}

Let $K$ be a compact Hausdorff space and let $C(K)$ denote the algebra of real
valued continuous functions on $K.$ Given a real Banach space $X,$ by $C(K,X)$
we denote the Banach space of continuous $X$-valued functions on $K.$ For each
$f\in C(K,X)$ by $|f|_{X}\in C(K)$ we denote its norm function. That is,
\[
|f|_{X}(t)=||f(t)||
\]
for each $t\in K.$ Furthermore $||f||=||\ |f|_{X}||_{C(K)}$ gives the norm on
the space $C(K,X).$ For each $a\in C(K),$ one defines the product $af\in
C(K,X)$ pointwise in the usual manner. With this product, $C(K,X)$ is a Banach
$C(K)$-module. For each $x\in X$ , we denote by $\overset{\rightarrow}{x}\in
C(K,X)$ the constant function defined by $\overset{\rightarrow}{x}(t)=x$ for
each $t\in K.$ Let $M(X)$ denote the submodule of $C(K,X)$ generated by the
constant functions. When $f\in C(K,X)$, its range $f(K)$ is compact. Therefore
$f(K)$ is a separable subset of $X.$ On the other hand, it is clear that the
range of each function in $M(X)$ is contained in a finite dimensional subspace
of $X$. Moreover if $f\in C(K,X)$ has its range contained in a finite
dimensional subspace of $X$ , then $f$ is in $M(X).$ That is, $f=\sum
a_{i}\overset{\rightarrow}{x}_{i}$ for $a_{i}\in C(K)$ and $x_{i}\in X$ with
$i=1,2\ldots n$ for some positive integer $n.$

\begin{lemma}
$M(X)$ is dense in $C(K,X).$
\end{lemma}

\begin{proof}
Let $f\in C(K,X)$ and $\varepsilon>0.$ There exist $\{x_{1,}x_{2}\ldots
x_{n}\}$ in $f(K)$ such that the open balls $B(x_{i},\varepsilon)$ in $X$ with
$i=1,2\ldots n$ cover $f(K).$ Let $U_{i}=f^{-1}(B(x_{i},\varepsilon))$ and
$t_{i}\in U_{i}$ such that $f(t_{i})=x_{i}$ with $i=1,2\ldots n.$ The
collection $\{U_{i}:i=1,2\ldots n\}$ is a finite open cover of $K.$ Let the
collection $\{a_{i}:i=1,2\ldots n\}$ in the positive unit ball of $C(K)$ be a
partition of unity subordinate to the cover $\{U_{i}:i=1,2\ldots n\}.$ Let
$g=\sum a_{i}\overset{\rightarrow}{x_{i}}.$ Then $||f-g||<\varepsilon.$
\end{proof}

Let $F$ be a Banach lattice with a quasi-interior point $u$ \cite[II.6.1]{S}.
Let its ideal center $Z(F)=C(K)$ both as an algebra and a Banach lattice.
Recall that an operator $T$ on $F$ is in the ideal center if for some
$\lambda>0,$ $|T(x)|\leq\lambda x$ for each non-negative $x\in F.$ That $F$
has a quasi-interior point $u$ means that $Z(F)u$ is dense in $F$
\cite[Proposition 1.1]{W1}. (Because of this fact a quasi-interior point is
also called a topological order unit.) Then for each $x\in F,$ the closure of
$Z(F)x$ is equal to the closed ideal $I(x)$ of $F$ generated by $x.$ Without
loss of generality we suppose $||u||_{F}=1.$ Then we have that $F$ is a Banach
$C(K)$-module. Also it is well known \cite[Theorem III.4.5]{S} that $F$ has a
functional representation as an ideal in $C_{\infty}(K),$ the set of extended
continuous functions on $K$ into $[-\infty,\infty].$ Recall that a continuous
function on $K$ into $[-\infty,\infty]$ is called an extended continuous
function if it is finite on a dense open subset of $K.$ For each $a\in C(K),$
$au$ in $F$ is represented by $a$ in $C_{\infty}(K).$ In particular $u$ is
represented by $1$ and $C(K)$ in $C_{\infty}(K)$ corresponds to the ideal
generated \ by $u$ in $F.$ In view of this we will always think of $F$ in
terms of its representation in $C_{\infty}(K)$ and replace the quasi-interior
point $u$ by the function $1.$

We introduce a new norm on $C(K,X)$ by%
\[
||f||_{F(X)}=||\text{ }|f|_{X}||_{F}%
\]
for each $f\in C(K,X).$ Let $F_{\pi}(X)$ denote the completion of $C(K,X)$ in
this norm. It is clear that $F_{\pi}(X)$ is also a Banach $C(K)$-module. Let
$f\in F_{\pi}(X)\smallsetminus C(K,X)$ and let $\{f_{n}\}$ be a sequence in
$C(K,X)$ that converges to $f$ in $F_{\pi}(X)$. Then the sequence of norm
functions $\{|f_{n}|_{X}\}$ is a Cauchy sequence in $F$ and converges to a
non-negative element of $F.$ We will denote the limit point by $|f|_{X}$ and
call it the norm function of $f\in F_{\pi}(X).$ We are justified in this since
$||f||_{F(X)}=||$ $|f|_{X}||_{F}$ and $|f|_{X}\in C_{\infty}(K).$

It is well known (see e.g., \cite{V}, \cite[Chapter V.3]{S}, \cite{K}) that
each cyclic subspace of a Banach $C(K)$-module $M$ may be represented as a
Banach lattice with quasi-interior point. Namely, for each $x\in M,$ the
cyclic subspace $M(x)=cl(C(K)x)$ is a Banach lattice with positive cone
$cl(C(K)_{+}x)$ and quasi-interior point $x.$ In general, the ideal center of
$M(x)$ is $w$-$cl(C(K)_{|M(x)})$ (see e.g., \cite[Theorem 1]{O1}). Here `$cl$'
denotes closure in norm in $M$ and `$w$-$cl$' the closure in the weak operator
topology when we consider the representation of $C(K)$ as an algebra of operators on $M(x).$ In the case of $F_{\pi}(X),$ we get more precise
information on the structure of its cyclic subspaces.

\begin{lemma}
For each $f\in F_{\pi}(X),$ the cyclic subspace $F(f)$ is isometric and
lattice isomorphic to the closed ideal $I(|f|_{X})$ of $F.$ In particular, for
each $0\neq x\in X,$ the cyclic subspace $F(\overset{\rightarrow}{x})$ is
isometric and lattice isomorphic to $F$.
\end{lemma}

\begin{proof}
Let $f\in C(K,X)$ and $a\in C(K).$ Then, for each $t\in K,$
\[
|af|_{X}(t)=||a(t)f(t)||=|a(t)|\text{ }||f(t)||=|a(t)||f|_{X}(t).
\]

Therefore $|af|_{X}=|a||f|_{X}.$ It follows that
\[
||af||_{F(X)}=||\text{ }|a||f|_{X}||_{F}=||a|f|_{X}||_{F}.
\]
Then the same equality also follows for $f\in F_{\pi}(X)\smallsetminus C(K,X)$
for any $a\in C(K).$ This proves that the submodule generated by $f$ in
$F_{\pi}(X)$ is isometric and lattice isomorphic to the sublattice
$C(K)|f|_{X}$ of $F.$ By passing to the closure we complete the first part of
the result. For any $x\in X,$ we have $|\overset{\rightarrow}{x}|_{X}=||x||1$
in $F.$ Since $1$ is a quasi-interior point, we have $F(\overset{\rightarrow
}{x})\cong F.$
\end{proof}

We will call an operator $T$ on a Banach $C(K)$-module $M,$ a\textbf{
multiplication operator} if $T(x)=ax$ for some $a\in C(K),$ for all $x\in M.$
For Banach lattices $F$ with quasi-interior point and ideal center
$Z(F)=C(K),$ the ideal center is maximal abelian (see e.g., \cite[Theorem
2.4]{W1}, \cite[Corollary 3]{O1}). That is, an operator on $F$ which commutes
with $C(K)$ is a multiplication operator.

\begin{theorem}
Suppose $K$ is a compact Hausdorff space and $F$ is a Banach lattice with
quasi-interior point $u$ and ideal center $Z(F)=C(K).$ Suppose $X$ is a Banach
space. Let $T$ be an operator on $F_{\pi}(X).$ Consider the following conditions:

\begin{enumerate}
\item[(i)] \textit{ }$T$\textit{ is a multiplication operator. }

\item[(iia)] \textit{ For each }$x\in X,$\textit{ the cyclic subspace
}$F(\overset{\rightarrow}{x})$\textit{ is left invariant by }$T.$\textit{ }

\item[(iib)] \textit{ For each }$x\in X,$\textit{ }$T$\textit{ commutes with
}$C(K)$\textit{ on }$F(\overset{\rightarrow}{x}).$
\end{enumerate}

Then (i) $\Leftrightarrow$\textit{(iia) and (iib).}
\end{theorem}

\begin{proof}
The implication (i)$\Rightarrow$(iia) and (iib) is clear. Conversely assume
that (iia) and (iib) hold. Let $x\in X.$ From Lemma 2, we have that
$F(\overset{\rightarrow}{x})$ is isometric and lattice isomorphic to $F.$ Then
(iia) implies that we may consider $T$ restricted to $F(\overset{\rightarrow
}{x})$ as an operator on $F.$ Then (iib) and the discussion preceeding the
theorem imply that $T$ is a multiplication operator on $F(\overset
{\rightarrow}{x}).$ That is there exists $a_{x}\in C(K)$ such that
$T(f)=a_{x}f$ for all $f\in F(\overset{\rightarrow}{x}).$ Suppose $x,y\in X$
are linearly independent. Then%
\[
a_{x}\overset{\rightarrow}{x}+a_{y}\overset{\rightarrow}{y}=T(\overset
{\rightarrow}{x}+\overset{\rightarrow}{y})=a_{x+y}(\overset{\rightarrow}%
{x}+\overset{\rightarrow}{y}).
\]

Therefore%
\[
(a_{x}-a_{x+y})\overset{\rightarrow}{x}=(a_{x+y}-a_{y})\overset{\rightarrow
}{y}.
\]

From the linear independence of $x,y,$ it easily follows that
\[
(a_{x}-a_{x+y})=(a_{x+y}-a_{y})=0.
\]

Hence there exists an $a\in C(K)$ such that $T(f)=af$ for all $f\in M(X).$ By
Lemma 1, $M(X)$ is dense in $C(K,X)$ and by construction $C(K,X)$ is dense in
$F_{\pi}(X).$ This means that $M(X)$ is dense in $F_{\pi}(X).$ Therefore by
continuity $T$ is multiplication by $a$ on all of $F_{\pi}(X).$
\end{proof}

We point out that Lemma 2 and Theorem 3 hold for $C(K,X).$ Namely let
$F=C(K),$ then $F_{\pi}(X)=C(K,X).$

\section{K\"{o}the-Bochner spaces}

Let $(\Omega,\Sigma,\mu)$ be a complete probability measure space. Let
$L(\mu)$ denote the Dedekind complete vector lattice of the equivalence
classes of almost everywhere finite measurable functions over $(\Omega
,\Sigma,\mu)$. The ideal center of $L(\mu)$ is $L^{\infty}(\mu)=C(S).$ Due to
the fact that $L^{1}(\mu)^{\ast}=L^{\infty}(\mu),$ the space $S$ is
hyperstonian. The characteristic functions of the sets of positive measure in
$\Sigma$ correspond to the characteristic functions of the clopen subsets of
$S.$ This correspondance gives the identification between $L^{\infty}(\mu)$
and $C(S).$ We will also assume the identification of $L(\mu)$ with
$C_{\infty}(S).$ We denote by $E$ a Banach function space in $L(\mu).$ We
assume that $1\in E$ with unit norm, that $E$ is an ideal in $L^{1}(\mu),$ and
that this embedding is continuous. Clearly $E$ is Dedekind complete and $1$ is
a weak order unit in $E.$ We let $F=I(1),$ the Banach function space given by
the closed ideal generated by $1$ in $E.$ Then $1$ is a quasi-interior point
of $F.$ $L^{\infty}(\mu)$ is the ideal center of both $E$ and $F.$ Let $X$ be
a Banach space. On $X,$ we take the Borel $\sigma$-algebra with respect to the
norm. By $L(\mu,X)$ we denote the $L^{\infty}(\mu)$-module consisting of the
equivalence classes of the strongly measurable functions on $\Omega$ into $X$.
That is, $f\in L(\mu,X)$ means $f:\Omega\rightarrow X$ is measurable and
$f(\Omega\smallsetminus A)$ is separable in $X$ for some $A\in\Sigma$ with
$\mu(A)=0$ . As in the previous section $|f|_{X}\in L(\mu)$ will denote the
norm function of a strongly measurable function $f.$ The
\textbf{K\"{o}the-Bochner} space $E(X)$ is defined (see e.g. \cite[Section
3.4]{P}) to consist of all functions $f\in L(\mu,X)$ with $|f|_{X}\in E.$ It
can be shown that $E(X)$ is a Banach $L^{\infty}(\mu)$-module when it is
equipped with the norm%
\[
||f||_{E(X)}=||\text{ }|f|_{X}||_{E}\text{ .}%
\]

When $E$ is $L^{1}(\mu)$ or $L^{\infty}(\mu),$ as usual, we denote $E(X)$ by
$L^{1}(\mu,X)$ and $L^{\infty}(\mu,X)$ respectively. $L^{1}(\mu,X)$ is called
the space of Bochner-integrable functions \cite[II.2.2]{DU}. It is clear from
the definition that we always have the continuous embeddings
\[
L^{\infty}(\mu,X)\subset E(X)\subset L^{1}(\mu,X).
\]
We denote by $S(X)$ the functions with finite range and by $\sigma$-$S(X)$ the
functions with countable range in $L(\mu,X).$ Let $\{e_{i}:i\in\mathcal{I}\}$
be a set of pairwise disjoint idempotents in $L^{\infty}(\mu)$ and let
$\{x_{i}:i\in\mathcal{I\}}$ be a set of distinct elements in $X$ for some
index set $\mathcal{I}$. Then if $\mathcal{I}$ is finite $f=\sum e_{i}%
\overset{\rightarrow}{x_{i}}$ is in $S(X)$ and if $\mathcal{I}$ is at most
countable $f$ is in $\sigma$-$S(X).$ Conversely each element of $S(X)$ or
$\sigma$-$S(X)$ may be written in this prescribed form. We include the proof
of the next lemma for the sake of completeness \cite[II.1.3]{DU}.

\begin{lemma}
Suppose $f\in L(\mu,X)$ and $\varepsilon>0.$ Then there exists $g\in\sigma
$-$S(X)$ such that $||f-g||_{\infty}<\varepsilon.$
\end{lemma}

\begin{proof}
Let $A$ be a set of measure zero in $\Omega$ such that $f(\Omega\smallsetminus
A)$ is separable in $X.$ Then for any $\varepsilon>0$ there is a countable
collection $\{B(x_{i},\frac{\varepsilon}{2}):x_{i}\in X,$ $i=1,2,\ldots\}$ of
open balls that covers $f(\Omega\smallsetminus A).$ Let the coressponding
collection of measurable sets $\{U_{i}=f^{-1}(B(x_{i},\frac{\varepsilon}%
{2})):i=1,2,...\}$ be given in $\Omega\smallsetminus A.$ Define a disjoint
sequence of measurable sets in $\Omega\smallsetminus A$ as follows:
$V_{1}=U_{1},$and $V_{n+1}=U_{n+1}\smallsetminus\underset{i=1}{\overset
{n}{\cup}}U_{i}$ for each $n=1,2,\ldots.$ It is clear that $\{V_{i}\}$ covers
$\Omega\smallsetminus A.$ Some of the sets in the sequence may be empty in
which case we delete them from the sequence. Let $e_{i}\in L^{\infty}(\mu)$ be
the characteristic function of $V_{i.}$ Take $t_{i}\in V_{i}$ and let
$y_{i}=f(t_{i}).$ Let $g=\sum e_{i}\overset{\rightarrow}{y_{i}}.$ Then
$(f-g)\in L^{\infty}(\mu,X)$ and $||f-g||_{\infty}<\varepsilon.$
\end{proof}

We have the following consequence of Lemma 4.

\begin{lemma}
The subspace $E(X)\cap\sigma$-$S(X)$ is dense in the K\"{o}the-Bochner space
$E(X).$
\end{lemma}

\begin{proof}
Let $f\in E(X)$ and let $\varepsilon>0.$ By Lemma 4, there is $g\in\sigma
$-$S(X)$ such that $(g-f)\in L^{\infty}(\mu,X)$ and $||g-f||_{\infty
}<\varepsilon.$ We have $|g|_{X}\leq|g-f|_{X}+|f|_{X}$ and since $L^{\infty
}(\mu)\subset E$, it follows that $g\in E(X)$ and $||g-f||_{E(X)}%
<\varepsilon.$
\end{proof}

We will make a brief comparison of the spaces defined in the previous section
and the K\"{o}the-Bochner spaces. Note that when we have $E=L^{\infty}%
(\mu)=C(S)=F.$ However, when $X$ is infinite dimensional, $C(S,X)$ is a proper
closed subspace of $L^{\infty}(\mu,X).$ On the other hand the opposite is also
possible. When $E=L^{1}(\mu)=cl(L^{\infty}(\mu)1)=F,$ we have $L^{1}%
(\mu,X)=L_{\pi}^{1}(\mu,X).$ This follows from the fact that $S(X)$ is dense
in $L^{1}(\mu,X)$ and that $M(X)$ is dense in $L_{\pi}^{1}(\mu,X)$ with
$S(X)\subset M(X).$ In the case when $F=I(1)$ is a proper ideal of $E,$ we
have the isometric inclusions%
\[
F_{\pi}(X)\subset F(X)\subset E(X).
\]
The embedding of $F(X)$ into $E(X)$ is always proper. However, if $L^{\infty
}(\mu,X)\subset F_{\pi}(X)$ or equivalently if $F$ has order continuous norm,
then $F_{\pi}(X)=F(X)$ . Otherwise that inclusion is also proper (see Remark
15 and Remark 17).

When $S$ is a Stonian compact Hausdorff space, a $C(S)$-module $M$ is called a
\textbf{Kaplansky module, } \cite{Ka} if for each $x\in M,$ the set $\{a\in
C(S):ax=0\}$ is a band in $C(S).$ Every Dedekind complete vector lattice is a
Kaplansky module over its ideal center.

\begin{lemma}
$E(X)$ is a Kaplansky $L^{\infty}(\mu)$-module.
\end{lemma}

\begin{proof}
Suppose $a_{\alpha}\in L_{+}^{\infty}(\mu)$ for all $\alpha\in\mathcal{I}$ and
$\sup a_{\alpha}=a\in L^{\infty}(\mu).$ Also suppose, for some $f\in E(X),$ we
have $a_{\alpha}f=0$ for all $\alpha\in\mathcal{I}.$ Then $||a_{\alpha
}f||_{E(X)}=0$ implies $a_{\alpha}|f|_{X}=0$ for each $\alpha\in\mathcal{I}$
in $E.$ Since $E$ is a Kaplansky module, we have $a|f|_{X}=0.$ Hence $af=0$
and $E(X)$ is a Kaplansky module.
\end{proof}

The following theorem is an analogue of Theorem 3 for K\"{o}the-Bochner
spaces. Note that it is weaker then Theorem 3, since (iic) requires more of
the operator $T$ than (iib).

\begin{theorem}
Let $E(X)$ be the K\"{o}the-Bochner space associated with a Banach function
space $E$ on the probability measure space $(\Omega,\Sigma,\mu),$ and a Banach
space $X.$ Let $T$ be an operator on $E(X).$ Consider the following conditions:

\begin{enumerate}
\item[(i)] \textit{ }$T$\textit{ is a multiplication operator. }

\item[(iia)] \textit{ For each }$x\in X,$\textit{ the cyclic subspace
}$E(\overset{\rightarrow}{x})$\textit{ is left invariant by }$T.$\textit{ }

\item[(iic)] \textit{ }$T$\textit{ commutes with }$L^{\infty}(\mu)$ on $E(X).$
\end{enumerate}

Then (i) $\Leftrightarrow$\textit{(iia) and (iic).}
\end{theorem}

\begin{proof}
We only need to prove (iia) and (iic)$\Rightarrow$(i). Let $F=I(1)$ in $E.$ As
we noted above the definition of the norms show $F_{\pi}(X)$ is a closed
subspace of $E(X).$ Also, for each $x\in X,$ we have $E(\overset{\rightarrow
}{x})=cl(L^{\infty}(\mu)\overset{\rightarrow}{x})=F(\overset{\rightarrow}%
{x})\subset F_{\pi}(X).$ Then condition (iia), in particular, implies that
$T(a\overset{\rightarrow}{x})\in F_{\pi}(X)$ for all $a\in L^{\infty}(\mu)$
and $x\in X.$ That is, $T$ maps $M(X)$, the submodule of $E(X)$ generated by
the constant functions, into $F_{\pi}(X).$ Since $T$ is bounded and $M(X)$ is
dense in $F_{\pi}(X)$ by construction, when we pass to closure in
$T(M(X))\subset F_{\pi}(X),$ we obtain that $F_{\pi}(X)$ is invariant under
$T.$ Then (iia), (iic) and Theorem 3 imply that $T$ is a multiplication
operator on $F_{\pi}(X).$ So there is an $a\in L^{\infty}(\mu)$ such that
$T(f)=af$ for all $f\in F_{\pi}(X).$ Let $g\in E(X)\cap\sigma$-$S(X).$ Suppose
that $\{e_{i}\}$ is a sequence of disjoint idempotents in $L^{\infty}(\mu)$
and $\{x_{i}\}$ is a sequence of distinct elements of $X$ such that $g=\sum
e_{i}\overset{\rightarrow}{x_{i}}.$ We will assume without loss of generality
that $\sup e_{i}=1.$ For each positive integer $n,$ let $\chi(n)=\underset
{1\leq i\leq n}{\sum}e_{i}\in L^{\infty}(\mu).$ Then $\chi(n)g\in S(X)\subset
M(X)\subset F_{\pi}(X),$ for each $n.$ Hence, by (iic) we have, for each $n,$%
\[
\chi(n)T(g)=T(\chi(n)g)=a(\chi(n)g)=\chi(n)(ag).
\]

Since $\sup\chi(n)=\sup e_{i}=1$ and $E(X)$ is a Kaplansky module, we have
$T(g)=ag$ for all $g\in E(X)\cap\sigma$-$S(X).$ By Lemma 5, we have that
$E(X)\cap\sigma$-$S(X)$ is dense in $E(X).$ Therefore $T$ is multiplication by
$a\in L^{\infty}(\mu)$ on all of $E(X).$
\end{proof}

In the next section we will consider the application of Theorem 7 to a
functional equation that Calabuig, Rodr\'{i}guez, and S\'{a}nchez-P\'{e}rez  defined in
\cite{CRS} and used as a criteria for identifying multiplication operators on
K\"{o}the-Bochner spaces.

\section{The functional equation}

Let $T$ be an operator on the K\"{o}the-Bochner space $E(X).$ In \cite{CRS}
the following functional equation was considered :
\[
T(\langle af,x^{\ast}\rangle\overset{\rightarrow}{x})=a\langle T(f),x^{\ast
}\rangle\overset{\rightarrow}{x}%
\]
for all $a\in L^{\infty}(\mu),f\in E(X),x\in X$ and $x^{\ast}\in X^{\ast}.$
Here $\langle x,x^{\ast}\rangle$ denotes the action of the Banach dual
$X^{\ast}$ on the Banach space $X.$ In the functional equation, however,
$\langle f,x^{\ast}\rangle\in E$ and
\[
\langle f,x^{\ast}\rangle(t)=\langle f(t),x^{\ast}\rangle
\]
for all $t\in\Omega\smallsetminus A$ for some set $A$ of measure zero in
$\Omega.$ The authors show in \cite[Corollary 2.3]{CRS} that when $E$ has
order continuous norm, $T$ is a multiplication operator if and only if $T$
satisfies the functional equation. We intend to apply Theorem 3 and Theorem 7
to extend the result to operators on spaces $F_{\pi}(X)$ and to operators on
all K\"{o}the-Bochner spaces $E(X).$

Initially we need to show that the terms of the equation makes sense in
$F_{\pi}(X).$ Let $f\in C(K,X)$ then it is clear that $\langle f,x^{\ast
}\rangle$ makes sense when defined pointwise at each $t\in K$ and we have
$\langle f,x^{\ast}\rangle\in C(K)\subset F.$ For each $t\in K$ $,$ we have
$|\langle f(t),x^{\ast}\rangle|\leq||f(t)||||x^{\ast}||.$ Then in $F,$ we have
$|\langle f,x^{\ast}\rangle|\leq|f|_{X}||x^{\ast}||.$ Therefore for a fixed
$x^{\ast}\in X^{\ast}$ and for all $f\in C(K,X)\subset F_{\pi}(X),$ we get
$||\langle f,x^{\ast}\rangle||_{F}\leq||f||_{F(X)}||x^{\ast}||.$ This means
that for each $x^{\ast}\in X^{\ast},$ the mapping $f\leadsto\langle f,x^{\ast
}\rangle$ defines a bounded linear transformation on the subspace $C(K,X)$ of
$F_{\pi}(X)$ into\ $F.$ Since $C(K,X)$ is dense in $F_{\pi}(X),$ the map
extends uniquely to all of $F_{\pi}(X).$ For each $f\in F_{\pi}%
(X)\smallsetminus C(K,X)$ let $\langle f,x^{\ast}\rangle\in F$ denote its
image under the extension. Also for each $x\in X,$ by Lemma 2 we have
$F(\overset{\rightarrow}{x})\cong F.$ Let $\langle f,x^{\ast}\rangle
\overset{\rightarrow}{x}\in F(\overset{\rightarrow}{x})$ denote the image of
$||x||\langle f,x^{\ast}\rangle\in F$ in this isomorphism. Now it is clear that we
can write the functional equation in $F_{\pi}(X).$

\begin{lemma}
Let $T$ be an operator on $F_{\pi}(X).$ Suppose $T$ satisfies
\[
\text{(iii) }T(a\langle f,x^{\ast}\rangle\overset{\rightarrow}{x})=a\langle
T(f),x^{\ast}\rangle\overset{\rightarrow}{x}%
\]
for all $a\in C(K),f\in F_{\pi}(X),x\in X$ and $x^{\ast}\in X^{\ast}.$ Then

\textit{(iia) For each }$x\in X,$\textit{ the cyclic subspace }$F(\overset
{\rightarrow}{x})$\textit{ is left invariant by }$T.$\textit{ }

\textit{(iib) For each }$x\in X,$\textit{ }$T$\textit{ commutes with }%
$C(K)$\textit{ on }$F(\overset{\rightarrow}{x}).$
\end{lemma}

\begin{proof}
Let $a\in C(K)$ and $x\in X.$ Take $x^{\ast}\in X^{\ast}$ such that $\langle
x,x^{\ast}\rangle=1.$ Then%
\[
T(a\overset{\rightarrow}{x})=T(a\langle\overset{\rightarrow}{x},x^{\ast
}\rangle\overset{\rightarrow}{x})=a\langle T(\overset{\rightarrow}{x}%
),x^{\ast}\rangle\overset{\rightarrow}{x}.
\]

By the discussion preceeding the lemma, since $\langle T(\overset{\rightarrow
}{x}),x^{\ast}\rangle\in F,$ we have $\langle T(\overset{\rightarrow}%
{x}),x^{\ast}\rangle\overset{\rightarrow}{x}\in F(\overset{\rightarrow}{x}).$
Therefore $T(C(K)\overset{\rightarrow}{x})\subset F(\overset{\rightarrow}%
{x}).$ This implies (iia). Let $a,b\in C(K)$ and $x\in X.$ Choose $x^{\ast}\in
X^{\ast}$ such that $\langle x,x^{\ast}\rangle=1.$ Then%
\[
T(ab\overset{\rightarrow}{x})=T(ab\langle\overset{\rightarrow}{x},x^{\ast
}\rangle\overset{\rightarrow}{x})=a(b\langle T(\overset{\rightarrow}%
{x}),x^{\ast}\rangle\overset{\rightarrow}{x})=aT(b\langle\overset{\rightarrow
}{x},x^{\ast}\rangle\overset{\rightarrow}{x})=aT(b\overset{\rightarrow}{x}).
\]
When we pass to closure, we have (iib).
\end{proof}

Now we can state the following corollary to Theorem 3 and prove it by using
Lemma 8.

\begin{corollary}
Let $T$ be an operator on $F_{\pi}(X).$ Then the following are equivalent:

\textit{(i) }$T$\textit{ is a multiplication operator.}

\textit{(iii) The equality}%
\[
\text{ }T(a\langle f,x^{\ast}\rangle\overset{\rightarrow}{x})=a\langle
T(f),x^{\ast}\rangle\overset{\rightarrow}{x}%
\]
\textit{holds for all }$a\in C(K),f\in F_{\pi}(X),x\in X$\textit{ and
}$x^{\ast}\in X^{\ast}.$
\end{corollary}

Next we consider the functional equation in the case of K\"{o}the-Bochner
spaces. In this case difficulties arise when $1\in E$ is not a quasi-interior
point in $E.$ With $F=I(1)$, suppose $\phi\in E\smallsetminus F.$ Then, for
each $x\in X,$ we have $\phi\overset{\rightarrow}{x}\notin E(\overset
{\rightarrow}{x})=F(\overset{\rightarrow}{x})$ by Lemma 2. So for each $x\in
X,$ we need to consider the submodule of $E(X)$ that is given by
$[\overset{\rightarrow}{x}]:=\{\phi\overset{\rightarrow}{x}:\phi\in E\}.$ We
call $[\overset{\rightarrow}{x}]$ the \textbf{band-type submodule} generated
by $\overset{\rightarrow}{x}\in E(X).$

\begin{lemma}
Let $E(X)$ be a K\"{o}the-Bohner space and let $x\in X.$

\begin{enumerate}
\item The band-type submodule $[\overset{\rightarrow}{x}]$ is isometric and
lattice isomorphic to the Banach function space $E.$

\item Suppose that for some $f\in E(X),$ there is a family of upwards directed
idempotents $\{e_{\alpha}:\alpha\in\mathcal{I}\}$ in $L^{\infty}(\mu)$ such
that $\sup e_{\alpha}=1$ and $e_{\alpha}f\in\lbrack\overset{\rightarrow}{x}]$
for each $\alpha\in\mathcal{I}.$ Then $f\in\lbrack\overset{\rightarrow}{x}].$
\end{enumerate}
\end{lemma}

\begin{proof}
Clearly $[\overset{\rightarrow}{x}]$ is a submodule of $E(X)$ and it is clear
that $E\subset L(\mu)$ induces on $[\overset{\rightarrow}{x}]$ a vector
lattice structure with respect to which it is a Dedekind complete vector
lattice with weak order unit $\overset{\rightarrow}{x}.$ Let $\phi\in E.$
Then
\[
||\phi\overset{\rightarrow}{x}||_{E(X)}=||\text{ }|\phi\overset{\rightarrow
}{x}|_{X}||_{E}=||\text{ }|\phi|||x||\text{ }||_{E}=||\phi||_{E}||x||=||\text{
}||x||\phi||_{E}.
\]
Hence the correspondence $\phi\overset{\rightarrow}{x}\longleftrightarrow
||x||\phi$ gives the isometric lattice isomorphism between the two spaces.
This completes part (1). To prove part (2), assume the conditions in the
statement of (2). Then , there is a collection of functions $\{\phi_{\alpha
}:\alpha\in\mathcal{I}\}$ in $E$ such that $e_{\alpha}f=\phi_{\alpha}%
\overset{\rightarrow}{x}$ for each $\alpha\in\mathcal{I}$ and $e_{\alpha}%
\phi_{\beta}=\phi_{\alpha}$ whenever $\alpha\leq\beta$ in the order of the
index set $\mathcal{I}.$ Then $e_{\alpha}{|f}|_{X}%
=|\phi_{\alpha}|||x||$ for each $\alpha\in\mathcal{I}$. It follows that
$||x||(\sup|\phi_{\alpha}|)=|f|_{X}.$ Moreover, there exists an idempotent
$e\in L^{\infty}(\mu)$ such that for all $\alpha\in\mathcal{I},$ we have
$e|\phi_{\alpha}|=\phi_{\alpha}^{+}$ and $(1-e)|\phi_{\alpha}|=\phi_{\alpha
}^{-}.$ Let $a=e-(1-e).$ Hence $e_{\alpha}a|f|_{X}=a|\phi_{\alpha}%
|||x||=\phi_{\alpha}||x||$ and $e_{\alpha}f=\phi_{\alpha}\overset{\rightarrow
}{x}=e_{\alpha}(\frac{a}{||x||}|f|_{X})\overset{\rightarrow}{x}$ for each
$\alpha\in\mathcal{I}.$ Since $E(X)$ is a Kaplansky module (Lemma 6), we have
\[
f=\frac{a}{||x||}|f|_{X}\overset{\rightarrow}{x}\in\lbrack\overset
{\rightarrow}{x}].
\]
This proves part (2).
\end{proof}

\begin{lemma}
Let $T$ be an operator on the K\"{o}the-Bochner space $E(X).$ Suppose $T$
satisfies
\[
\text{(iii) }T(a\langle f,x^{\ast}\rangle\overset{\rightarrow}{x})=a\langle
T(f),x^{\ast}\rangle\overset{\rightarrow}{x}%
\]
for all $a\in L^{\infty}(\mu),f\in F_{\pi}(X),x\in X$ and $x^{\ast}\in
X^{\ast}.$ Then

\textit{(iia) For each }$x\in X,$\textit{ the cyclic subspace }$E(\overset
{\rightarrow}{x})$\textit{ is left invariant by }$T,$\textit{ }

\textit{(iic) }$T$\textit{ commutes with }$L^{\infty}(\mu)$\textit{ on
}$E(X).$
\end{lemma}

\begin{proof}
We will prove (iic) first. Observe that $a\langle f,x^{\ast}\rangle=\langle
af,x^{\ast}\rangle$ for all $a\in L^{\infty}(\mu),f\in E(X)$ and $x^{\ast}\in
X^{\ast}.$ Then for each $x\in X,$ (iii) implies%
\[
T(a\langle f,x^{\ast}\rangle\overset{\rightarrow}{x})=a\langle T(f),x^{\ast
}\rangle\overset{\rightarrow}{x}=\langle aT(f),x^{\ast}\rangle\overset
{\rightarrow}{x}%
\]
and%
\[
T(a\langle f,x^{\ast}\rangle\overset{\rightarrow}{x})=T(\langle af,x^{\ast
}\rangle\overset{\rightarrow}{x})=\langle T(af),x^{\ast}\rangle\overset
{\rightarrow}{x}%
\]
for all $a\in L^{\infty}(\mu),f\in E(X)$ and $x^{\ast}\in X^{\ast}.$ Then by
Lemma 10 part(1), we have $\langle aT(f),x^{\ast}\rangle=\langle
T(af),x^{\ast}\rangle$ for all $x^{\ast}\in X$ when we take some fixed $a\in
L^{\infty}(\mu),f\in E(X).$ Since the functions in $E(X)$ have separable
range, by an often used result \cite[II.2.7]{DU}, we have $aT(f)=T(af).$ That
is (iii) implies (iic). Next we will show that for each $x\in X,$ the operator
$T$ leaves the band-type submodule $[\overset{\rightarrow}{x}]$ invariant. Let
$\phi\in E.$ Consider the sequence of measurable sets $V_{n}=\{t:|\phi(t)|\leq
n\},$ $n=1,2,\ldots.$ Let $e_{n}$ denote the characteristic function of
$V_{n}.$ Then the sequence $\{e_{n}\}$ is increasing and $\sup e_{n}=1$ in
$L^{\infty}(\mu).$ We have $e_{n}\phi\in L^{\infty}(\mu),$ for each $n.$ Also
take $x^{\ast}\in X^{\ast}$ such that $\langle x,x^{\ast}\rangle=1.$ Then, by
(iic),%
\[
e_{n}T(\phi\overset{\rightarrow}{x})=T(e_{n}\phi\overset{\rightarrow}%
{x})=T(e_{n}\phi\langle\overset{\rightarrow}{x},x^{\ast}\rangle\overset
{\rightarrow}{x})=e_{n}\phi\langle T(\overset{\rightarrow}{x}),x^{\ast}%
\rangle\overset{\rightarrow}{x}.
\]
That is $e_{n}T(\phi\overset{\rightarrow}{x})\in\lbrack\overset{\rightarrow
}{x}],$ for each $n=1,2,\ldots.$ Then Lemma 10 part (2) implies that
$T(\phi\overset{\rightarrow}{x})\in\lbrack\overset{\rightarrow}{x}].$ That is
, $T$ leaves $[\overset{\rightarrow}{x}]$ invariant. Furthermore (iic) implies
$T$ commutes with $L^{\infty}(\mu)$ on $[\overset{\rightarrow}{x}].$ By Lemma
10 part (1), we have $E\cong\lbrack\overset{\rightarrow}{x}].$ Therefore $T$
restricted to $[\overset{\rightarrow}{x}]$ corresponds to an operator on $E$
that commutes with the ideal center $Z(E)=L^{\infty}(\mu).$ Since $E$ is a
Dedekind complete Banach lattice, it is well known that its ideal center is
maximal abelian (e.g., \cite[Proposition 2.1]{W1}). So $T$ is a multiplication
operator on $E$ and therefore on $[\overset{\rightarrow}{x}].$ That is for
some $a_{x}\in L^{\infty}(\mu),$ for all $\phi\in E,$ we have $T(\phi
\overset{\rightarrow}{x})=a_{x}\phi\overset{\rightarrow}{x}.$ When $a\in
L^{\infty}(\mu)\subset E,$ we have $T(a\overset{\rightarrow}{x})=a_{x}%
a\overset{\rightarrow}{x}\in E(\overset{\rightarrow}{x}).$ Therefore $T$
leaves the cyclic subspace $E(\overset{\rightarrow}{x})$ invariant. That is
(iii) implies (iia).
\end{proof}

The use of Lemma 11 and Theorem 7 yield the following corollary to Theorem 7.

\begin{corollary}
Let $T$ be an operator on the K\"{o}the-Bochner space $E(X).$ Then the
following are equivalent:

\textit{(i) }$T$\textit{ is a multiplication operator.}

\textit{(iii) The equality}%
\[
\text{ }T(a\langle f,x^{\ast}\rangle\overset{\rightarrow}{x})=a\langle
T(f),x^{\ast}\rangle\overset{\rightarrow}{x}%
\]
\textit{holds for all }$a\in L^{\infty}(\mu),f\in E(X),x\in X$\textit{ and
}$x^{\ast}\in X^{\ast}.$
\end{corollary}

\section{Examples, Comlementary Results and Remarks}

Throught the section $E$ will be a Banach function space in $L^{1}(\mu)$ with
weak order unit $1$. $F$ will be the Banach function space given by the closed
ideal generated by $1$ in $E.$ We will think of $F_{\pi}(X)$ as the closure of
$S(X)$ in the K\"{o}the-Bochner space $E(X).$ For K\"{o}the-Bochner spaces, in
general, the conditions of Theorem 3 (i.e.,  (iia) and (iib)) are not sufficient
to identify the multiplication operators. We give below examples of operators
on $E(X)$ that satisfy the conditions (iia) and (iib) of Theorem 3 but are not
multiplication operators.

\begin{example}
\emph{{ Suppose that $F_{\pi}(X)$ is a proper closed subspace of $E(X).$ Let
$\mathcal{N}$ denote the non-zero operators on $E(X)$ that are zero on
$F_{\pi}(X).$ The set $\mathcal{N}$ is non-empty (via the Hahn-Banach
Theorem). Take $T\in\mathcal{N}$ and $a\in L^{\infty}(\mu)$ considered as a
multiplication operator on $E(X).$ Consider the operator $T+a$ on $E(X).$ When
restricted to $F_{\pi}(X)$ we have that $T+a=a.$ So $T+a$ is a multiplication
operator on $F_{\pi}(X)$ and leaves it invariant. Hence, by Theorem 3, $T+a$
satisfies the conditions (iia) and (iib) on $F_{\pi}(X).$ Recall, from the
proof of Theorem 7, that the cyclic subspace $E(\overset{\rightarrow}%
{x})=F(\overset{\rightarrow}{x})\subset F_{\pi}(X)$ for each $x\in X.$ That is
$T+a$ satisfies the conditions (iia) and (iib) on $E(X)$ but it is not a
multiplication operator on }}$E(X)$\emph{{. So all operators in the class
$\mathcal{N}+L^{\infty}(\mu)$ satisfy the conditions (iia) and (iib) on $E(X)$
but are not multiplication operators. When $X$ is infinite dimensional with
$E=L^{\infty}(\mu)=C(S)=F,$ we have that $F_{\pi}(X)=C(S,X)$ , $E(X)=L^{\infty
}(\mu,X)$ and $C(S,X)$ is a proper closed subspace of $L^{\infty}(\mu,X).$
Other examples are obtained when $F$ is a proper ideal of $E$ because in all
such cases $F_{\pi}(X)$ is a proper closed subspace of $E(X).$ See Remark 17
below for even more examples on which this type of operators may be
constructed. }}
\end{example}

Next we verify a statement made in Section 3 on K\"{o}the-Bochner spaces while
comparing the spaces $F(X)$ and $F_{\pi}(X).$ Initially we need to prove the
following result on K\"{o}the-Bochner spaces.

\begin{theorem}
Let $E$ be a Banach function space and let $X$ be an infinite dimensional
Banach space. Then $S(X)$ is dense in $E(X)$ if and only if $E$ has order
continuous norm.
\end{theorem}

\begin{proof}
It is well known that if $E$ has order continuous norm then $S(X)$ is dense in
$E(X)$ \cite{CRS}, \cite{P}. To complete the proof we will show that if $E$
does not have order continuous norm then $S(X)$ is not dense in $E(X).$ If $1$
is not a quasi-interior point of $E$ then there is nothing to prove. Since in
that case the closure of $S(X)$ in $E(X)$ equals $F_{\pi}(X)$ (with $F=I(1)$)
and $F_{\pi}(X)$ is a proper subspace of $E(X).$ Therefore let $1$ be a
quasi-interior point of $E$ and suppose $E$ does not have order continuous
norm. Then there is an increasing squence $\{\chi_{n}\}$ of idempotents in
$L^{\infty}(\mu)$ such that $\sup\chi_{n}=1$ but $\{\chi_{n}\}$ does not
converge to $1$ in norm in $E.$ This means that $\{\chi_{n}\}$ does not
converge in $E$ because the only limit it can have is $1.$ So $\{\chi_{n}\}$
is not a Cauchy sequence. Then for some $\delta>0$ we can choose a subsequence
$\{\chi_{k_{n}}\}$ such that $\delta\leq||\chi_{k_{n}}-\chi_{k_{m}}||_{E}$
when $n\neq m$ and $\sup\chi_{k_{n}}=1.$ Define a disjoint sequence of
idempotents by setting $e_{1}=\chi_{k_{2}},$ and $e_{i}=\chi_{k_{i+1}}%
-\chi_{k_{i}}$ for all $i\geq2.$ Clearly $\delta\leq||e_{i}||_{E}$ for all $i$
and $\sup\underset{i=1}{\overset{n}{%
{\textstyle\sum}
}}e_{i}=1.$ Since $X$ is infinite dimensional we can choose a distinct
sequence $\{x_{i}\}$ on the unit sphere of $X$ such that $||x_{i}-x_{j}||>1/2$
when $i\neq j.$ Now define a function $f\in L^{\infty}(\mu,X)$ by setting
$f=\sum e_{i}\overset{\rightarrow}{x}_{i}.$ Clearly $f\in E(X)$ and we will
show $f$ is not in the closure of $S(X).$ Given any $\varepsilon>0$ suppose
there is $g\in S(X)$ such that $||g-f||_{E(X)}<\varepsilon\delta.$ There is a
finite distinct subset $\{y_{k}:1\leq k\leq n\}$ in $X$ and a disjoint set of
idempotents $\{\gamma_{k}:1\leq k\leq n\}$ in $L^{\infty}(\mu)$ with $%
{\textstyle\sum}
\gamma_{k}=1$ such that $g=\sum\gamma_{k}\overset{\rightarrow}{y}_{k}.$ For
each $i$, we have
\[
e_{i}(g-f)=e_{i}(g-\overset{\rightarrow}{x}_{i})=e_{i}(\overset{n}%
{\underset{k=1}{%
{\textstyle\sum}
}}\gamma_{k}(\overset{\rightarrow}{y}_{k}-\overset{\rightarrow}{x}_{i})).
\]
Then%
\[
\underset{1\leq k\leq n}{\min}||y_{k}-x_{i}||\delta\leq\underset{1\leq k\leq
n}{\min}||y_{k}-x_{i}||\,||e_{i}||_{E}\leq||e_{i}(g-f)||_{E(X)}\leq
||g-f||_{E(X)}<\varepsilon\delta.
\]
Therefore for any $\varepsilon>0,$ and for each $i$ there exists some $k,$
$1\leq k\leq n$ such that $||y_{k}-x_{i}||<\varepsilon.$ But there are only
finitely many $y_{k}$, therefore, for some $y_{k},$ there must be a pair of
indices $i\neq j$ such that both $||y_{k}-x_{i}||<\varepsilon$ and
$||y_{k}-x_{j}||<\varepsilon$ hold. This implies $||x_{i}-x_{j}||<2\varepsilon
.$ When $\varepsilon\leq1/4,$ this contradicts our initial choice of the
sequence $\{x_{i}\}.$ Hence $f$ is not in the closure of $S(X)$ and the proof
is complete.
\end{proof}

\begin{remark}
It is clear from Theorem 14 that when $X$ is infinite dimensional, $F_{\pi
}(X)=F(X)$ if and only if $F$ has order continuous norm. The theorem also
shows that the main result of \cite[Theorem 1.4]{CRS} and its corollary
\cite[Corollary 2.3]{CRS} are equivalent.
\end{remark}

To verify a second statement made in Section 3 on K\"{o}the-Bochner spaces
while comparing the spaces $F(X)$ and $F_{\pi}(X),$ we need to prove the
following result.

\begin{theorem}
Let $E$ be a Banach function space and let $X$ be a Banach space. Suppose $1$
is a quasi-interior point of $E,$ then $L^{\infty}(\mu,X)$ is dense in $E(X).$
\end{theorem}

\begin{proof}
Let $g\in E(X)\cap\sigma$-$S(X).$ Suppose $g=\sum e_{i}\overset{\rightarrow
}{x_{i}}$ where $\{e_{i}\}$ is a sequence of disjoint idempotents in
$L^{\infty}(\mu)$ and $\{x_{i}\}$ is a sequence of distinct elements in $X.$
Let $\mathcal{I}(n)=\{i:||x_{i}||\leq n\}$ and let the sequence $\{g_{n}\}$ in
$L^{\infty}(\mu,X)$ be defined by
\[
g_{n}=\underset{i\in\mathcal{I}(n)}{\sum}e_{i}\overset{\rightarrow}{x_{i}%
}+\underset{i\notin\mathcal{I}(n)}{\sum}e_{i}(\frac{n}{||x_{i}||}%
\overset{\rightarrow}{x_{i}})
\]
for each $n=1,2\ldots.$ Note that
\begin{align*}
|g_{n}|_{X}  &  =|g|_{X}\wedge n1\text{ \ \ and}\\
|g-g_{n}|_{X}  &  =|g|_{X}-|g_{n}|_{X}\geq0
\end{align*}
for each $n=1,2\ldots.$ Since $1$ is a quasi-interior point in $E,$
$\{|g|_{X}\wedge n1\}$ converges to $|g|_{X}$ in $F.$ Then
\[
||g-g_{n}||_{E(X)}=||\text{ }|g-g_{n}|_{X}||_{E}=||\text{ }|g|_{X}-|g_{n}%
|_{X}||_{E}%
\]
implies that $\{g_{n}\}$ converges to $g$ in $E(X).$ Hence, by Lemma 5,
$L^{\infty}(\mu,X)$ is dense in $E(X).$
\end{proof}

\begin{remark}
Let $E$ be a Banach function space and let $F=I(1).$ Then , since $1$ is a
quasi-interior point of $F,$ by Theorem 16, we have $F_{\pi}(X)=F(X)$ if and
only if $L^{\infty}(\mu,X)\subset F_{\pi}(X).$ Note that in such a case, if
$X$ is infinite dimensional then Theorem 14 implies $F$ has order continuos
norm. Theorem 16 has certain simplifying implications when one looks at
examples. Namely suppose $\phi\in E$ such that $||\phi||_{E}=1$ and for some
$\varepsilon>0,$ $\varepsilon\leq\phi.$ Let $G=I(\phi)$ in $E.$ Since $G$ is a
Banach function space and $\phi$ is a quasi-interior in $G,$ we may consider
the spaces $G_{\pi}(X)$ and $G(X).$ However to do this we need to change the
functional representation of $G$. (So that we have $\phi$ correspond to $1.$)
On the other hand, by Lemma 2, one can see that for each $x\in X,$ the cyclic
subspace $G(\overset{\rightarrow}{x})$ in $G(X)$ may be identified with the
cyclic subspace $E(\phi\overset{\rightarrow}{x})$ in $E(X).$ In fact they are
both lattice isometric to $G.$ Then it follows that $G_{\pi}(X)\cong cl(\phi
C(S,X))$ in $E(X)$. Similarly if $f\in L^{\infty}(\mu,X),$ then the
definitions of the norms immediately give $||f||_{G(X)}=||\phi f||_{E(X)}.$
Since Theorem 16 gives that $G(X)$ is the closure of $L^{\infty}(\mu,X)$ in
the $G(X)$-norm, we have $G(X)\cong cl(\phi L^{\infty}(\mu,X))$ in $E(X).$ So
we may realize the spaces $G_{\pi}(X),$ $G(X)$ as submodules of $E(X)$ without
the need to change the functional representation of $G.$ In particular let $E$
be an Orlicz space whose Orlicz function does not satisfy the $\Delta_{2}%
$-condition, or let $E$ be a Marcinkiewicz space. Then, it is well known
\cite{KR} that $E$ does not have order continuous norm, and that $F=I(1)$ has
order continuous norm. However since $E$ does not have order continuous norm,
there must be $\phi\in E$ such that $G=I(\phi)$ does not have order continuous
norm. We may suppose without loss of generality that $||\phi||_{E}=1$ and for
some $\varepsilon>0,$ $\varepsilon\leq\phi$. Then the above discussion implies
that $G_{\pi}(X)$ is a proper subspace of $G(X)$ and this provides yet another
collection of K\"{o}the-Bochner spaces on which the operators of Example 13
may be constructed.
\end{remark}

In conclusion, we return to the discussion in the introduction and compare
Theorems 3 and 7 with the general result about multiplication operators on
Banach $C(K)$-modules (\cite[Theorem 6.2]{AAK}, \cite[Theorem 7]{HO}).

\begin{remark}
Let $M$ be a Banach $C(K)$-module and suppose that the set of multiplication
operators (i.e., multiplication by an element of $C(K)$) is closed in the
weak-operator topology in $\mathcal{L}(M),$ the set of (bounded) operators on
$M.$ Then we know that $T\in\mathcal{L}(M)$ is a multiplication operator if
and only if $T$ leaves invariant each cyclic subspace $M(x)$ for every $x\in
M.$ In the case of K\"{o}the-Bochner spaces as Banach $L^{\infty}(\mu
)$-modules we saw that it is sufficient for $T$ to satisfy the functional
equation in order to be a multiplication operator (Corollary 12). We also saw
in Theorems 3 and 7 that the conceptual basis of the functional equation
requires the operator to leave invariant a very restricted class of cyclic
subspaces, namely those given by the constant functions. This is compensated
by the commutativity conditions (iib) or (iic). However that the class of
multiplication operators should be weak-operator closed in $\mathcal{L}(E(X))$
is not mentioned because the constructions of the spaces $E(X)$ and $F_{\pi
}(X)$ guarantee that the set of multiplication operators is weak-operator
closed in both cases. If one takes subalgebras of $L^{\infty}(\mu)$ then the situation
changes. Let $E=L^{1}[0,1]$ with Lebesgue measure on $[0,1].$ Take the
subalgebra $C[0,1]$ of $L^{\infty}[0,1],$ the center of $E.$ Consider an
operator $T$ on $E(X)=L^{1}([0,1],X)$ such that, for each $x\in X$, $T$ leaves
$E(\overset{\rightarrow}{x})$ invariant and commutes with $C[0,1]$ on
$E(\overset{\rightarrow}{x}).$ Then $T\in L^{\infty}[0,1]$ (and not in
$C[0,1]$ in general). This works because $E$ has order continuous norm. Now
let $L^{\infty}[0,1]=C(S)$ where $S$ is hyperstonian and consider $C(S,X).$
Suppose $T$ is an operator on $C(S,X)$ that leaves the cyclic subspaces of
$C(S,X)$ generated by constant functions invariant and commutes with $C[0,1]$
on these cyclic subspaces. Then $T$ need not be a multiplication operator.
Initially note that, Lemma 2 implies that the cyclic subspaces of $C(S,X)$
generated by the constant functions are identified with $C(S)=L^{\infty
}[0,1].$ Then note that, a remarkable and deep result of de Pagter and Ricker
\cite{PR} states that the bicommutant of the mutiplication operators given by
$C[0,1]$ in $\mathcal{L}(L^{\infty}[0,1])$ consists of the bounded Riemann
integrable functions on $[0,1].$ So that an operator $A$ in $\mathcal{L}%
(L^{\infty}[0,1])$ that commutes with $C[0,1]$ is not in general a
multiplication operator in $L^{\infty}[0,1].$ However, using the methods in
the proof of Theorem 3, one can show that there is a single such operator
$A_{T}$ such that on $M(X)$ one has that
\[
T(%
{\textstyle\sum}
a_{i}\overset{\rightarrow}{x}_{i})=%
{\textstyle\sum}
A_{T}(a_{i})\overset{\rightarrow}{x}_{i}%
\]
for all $a_{i}\in L^{\infty}[0,1]$ and $x_{i}\in X$ with $i=1,2\ldots,n$ for
any $n.$ Then the operator is extended to all of $C(S,X)$ by density and continuity.
\end{remark}

\end{document}